\documentclass[mathpazo]{jpde}
\usepackage[OT1]{fontenc}
\usepackage{cite}
\usepackage{amsmath,amssymb, amsfonts}
\usepackage{graphicx}

\allowdisplaybreaks 

\theoremstyle{plain}
\newtheorem{Corollary}{Corollary}[section]

\makeatletter
\def\ps@firstpage{\ps@empty
  \def\@oddhead{\parbox{\textwidth}{\footnotesize
    JOURNAL OF PARTIAL DIFFERENTIAL EQUATIONS\hfill J.~Partial~Diff.~Eqs.\\
    Vol.~{\bfseries\thisvolume}, No.~\thisnumber, pp.~\thispagestart-\thispageend\hfill\thismonth~\thisyear}}
  \let\@evenhead\@oddhead
  \def\@oddfoot{\footnotesize http://www.global-sci.org/jpde/\hfill\thepage\hfill\mbox{}}
  \let\@evenfoot\@oddfoot
  \let\@mkboth\@empty}
\makeatother

\begin{document}
\hoffset -1.cm \voffset -1.1cm

\renewcommand\thisnumber{X}
\renewcommand\thisyear {XXX}
\renewcommand\thismonth{XXX}
\renewcommand\stitle{Blow-up on Locally Finite Graphs}
\renewcommand\thisvolume{X}

\title{Blow-up Solutions Of Parabolic $p$-Laplacian \\ Inequalities On Locally Finite Graphs}

\author[W. Y. Ma, L. Zhao]
{MA Wenyuan\corrauth{}, ZHAO Liang}
\address{ School of Mathematical Sciences, Key Laboratory of Mathematics and Complex Systems of MOE, Beijing Normal University, Beijing 100875, China.
\\
\\{\rm Received 29 November 2025; Accepted 3 February 2026}}
\emails{{\tt wenyuanma@mail.bnu.edu.cn} (W. Y. Ma)}

\begin{abstract}
  In this paper, we consider blow-up behavior of the parabolic $p$-Laplacian inequality with a nonlinear source $u_t - \Delta_p u \geq \sigma(x, t)\Phi(u)$ on a locally finite connected graph $G = (V, E)$.
  We prove a comparison principle for the inequality and use it to characterize the relationship between the initial data and the blow-up phenomena under varying growth rates of the nonlinearity $\Phi$. Specifically, we show that when the growth rate of $\Phi$ exceeds linear growth, solutions must exhibit finite-time blow-up under suitable initial conditions.
\end{abstract}

\clc{O175.2, O192}

\maketitle

\section{Introduction}
\label{introduction}

In recent years, research on partial differential equations or inequalities on graphs has attracted widespread attention \cite{Ge2017,Grillo2024a,Grigoryan2017,Grigoryan2016,Gu2022,Han2020,Lin2018,Lin2017,Liu2024b,Wu2024,Wu2021,Yang2024}.
Graphs are discrete mathematical structures composed of vertices and edges that effectively model relational systems between entities.
Owing to their powerful representational capacity, graph theory has found widespread applications across diverse domains involving relational analysis.
For typical applications, graphs serve as powerful analytical tools to model user interaction patterns in social networks, optimize distribution networks in logistics systems, and characterize molecular interaction networks in computational biology.
Furthermore, graphs also provide a discrete counterpart to continuous geometric spaces, including Euclidean spaces and Riemannian manifolds, making the study of differential equations on graphs a fundamentally important research area.

Regarding finite graphs, Xin, Xu and Mu \cite{Xin2014} established results on blow-up dynamics for the semilinear parabolic problem:
\begin{align}
  \begin{cases}
    u_t(x, t) = \Delta_{\omega} u(x, t) + u^p(x, t), & (x, t) \in S \times (0, +\infty),          \\
    u(x, t) = 0,                                     & (x, t) \in \partial S \times (0, +\infty), \\
    u(x, 0) = u_0(x),                                & x \in S \cup \partial S,
  \end{cases}
  \label{eq_xin2014}
\end{align}
where $S$ is a bounded domain with boundary $\partial S$ in a finite graph.
Their work demonstrates that, when $p \leq 1$, the solution of \eqref{eq_xin2014} exists globally, while for $p > 1$, nonnegative nontrivial solutions exhibit finite-time blow-up with the blow-up rate $(T-t)^{-\frac{1}{p-1}}$.

Several recent studies \cite{Lin2018,Liu2024b,Wu2021} have investigated the semilinear parabolic equation
\begin{equation*}
    u_t - \Delta u = h(x)u^{1 + \alpha}.
\end{equation*}
on locally finite graphs. These works employed heat kernel estimates as the primary tool.
Under specified curvature conditions and volume growth assumptions, the authors investigated both blow-up phenomena and the existence of global solutions to the equation.

Moreover, Wu \cite{Wu2024} generalized the aforementioned results to the coupled parabolic system
\begin{align}
  \begin{cases}
    u_t(t, x) = \Delta u(t, x) + v^\alpha(t, x), & (t, x) \in (0, +\infty) \times V, \\
    v_t(t, x) = \Delta v(t, x) + u^\beta(t, x),  & (t, x) \in (0, +\infty) \times V, \\
    u(0, x) = u_0(x),                            & x \in V,                          \\
    v(0, x) = v_0(x),                            & x \in V,
  \end{cases}
  \label{eq_Wu}
\end{align}
by incorporating iterative comparison techniques. Under the curvature condition $CDE^{\prime}(n, 0)$ and polynomial volume growth of order $O(r^m)$, Wu demonstrated that when $\max\{\alpha,\beta\}+1 \geq \frac{m}{2}(\alpha\beta-1)$ or when the initial values $u_0$, $v_0$ are sufficiently large, all the non-negative solutions of \eqref{eq_Wu} blow up in finite time.

In addition to partial differential equations or systems, inequality equations or coupled systems of inequalities on graphs have also been extensively studied in recent years.
In \cite{Gu2022,minh2024,minh2025}, the authors employed the test function approach to study the existence and non-existence of solutions (Liouville-type theorems) for elliptic inequality equations (or systems) on weighted graphs, under specific assumptions such as volume growth conditions.

In this paper, we focus on a parabolic $p$-Laplacian inequality on a locally finite connected graph,
\begin{align}
  \begin{cases}
    u_t - \Delta_p u \geq \sigma(x, t)\Phi(u), & (x, t) \in V \times (0, \infty), \\
    u(x, 0) = u_0(x),                          & x \in V,
  \end{cases}
  \label{eq_mubiao}
\end{align}
where $V$ is the vertex set of the graph. We always assume $p > 2$ and the $p$-Laplacian will be defined in Section \ref{sec:preliminaries}.

We also assume that the following conditions are satisfied:
\begin{enumerate} 
  \item For any fixed $x \in V$, $\sigma(x, \cdot)$ is continuous in $[0, \infty)$, and $\sigma > 0$ for $(x, t) \in V \times [0, \infty)$;
  \item For any fixed $t \geq 0$, $\sigma(\cdot, t) \leq M$, where $M < \infty$ is a positive constant depending only on $t$;
  \item $\Phi$ is locally Lipschitz continuous in $\mathbb{R}$ with $\Phi(0) = 0$;
  \item The initial value $u_0 \geq 0$ and $u_0 \in \ell^{\infty}(V)$.
\end{enumerate}

Motivated by the work of \cite{Mastrolia2017a}, our research extends the investigation of parabolic $p$-Laplacian inequalities to graph settings.
Mastrolia, Monticelli and Punzo \cite{Mastrolia2017a} focused on investigating the parabolic differential inequalities with potential functions on Riemannian manifolds.
They established a novel non-existence theory for the parabolic inequality
\begin{equation*}
  u_t \geq \text{div}(|\nabla u|^{p-2}\nabla u) + V(x,t)u^q,
\end{equation*}
on a complete, non-compact $m$-dimensional Riemannian manifold, where $p > 1$, $q > \max\{p - 1, 1\}$, and the potential function $V > 0$.

Referring to the analytical techniques in \cite{Chung2017a}, we implement a comparison principle to establish an alternative framework for investigating existence and non-existence properties of the problem \eqref{eq_mubiao}.
Our main results include:

\begin{theorem}
  Assume that
  \begin{equation}
    \Phi(s) \geq (C_0 + \varepsilon_0) s^{p - 1}, \qquad    s \geq 0,
    \label{eq_HP}
  \end{equation}
  where $\varepsilon_0>0$ is a fixed constant and $C_0 > 0$ is a constant determined by $\sigma$ and $u_0$.
  Then, for initial value $u_0(x) \geq 0$ and $u_0(x) \not\equiv 0$, the solution of \eqref{eq_mubiao} blows up in finite time.
  \label{th_blowup_gpm1}
\end{theorem}

\begin{theorem}
  Assume that
  \begin{equation}
    \Phi(s) \geq C_1s^q, \qquad  s \geq 0,
    \label{eq_HP2}
  \end{equation}
  for some constant $C_1 > 0$ and $q > p - 1$.
  Then, for the initial value $u_0\geq 0$ with $\|u_0\|_{\ell^{\infty}(V)}$ sufficiently large, the solution of \eqref{eq_mubiao} blows up in finite time.
  \label{th_blowup_g1lpm1}
\end{theorem}

\begin{remark}
  The constant $C_0$ in Theorem \ref{th_blowup_gpm1} can be given a specific value in our proof of the theorem.
  In Theorem \ref{th_blowup_g1lpm1}, the phrase "$\|u_0\|_{\ell^\infty(V)}$ sufficiently large" means $\|u_0\|_{\ell^\infty(V)} \geq h_0$, where $h_0$ is defined in \eqref{eq_blowup_cauchy} below.
\end{remark}

We organize the paper as follows.
In Section \ref{sec:preliminaries}, we introduce some definitions and preliminary results that are indispensable for our main theorems.
In particular, we prove the comparison principle required for subsequent analysis.
In Section \ref{sec:blowup}, we establish the blow-up conditions and prove our main results.

\section{Preliminaries}
\label{sec:preliminaries}

Let $G = (V, E, \mu, \omega)$ (or simply $G = (V, E)$) be a locally finite, connected weighted graph, where $V$ denotes the vertex set, $E\subseteq V \times V$ represents the edge set, with $(x,x)\notin E$ for all $x\in V$.
The weight function $\omega : V \times V \to [0, \infty)$ satisfies
\begin{enumerate} 
  \item $\omega_{xy} = \omega_{yx}$, for all $x, y \in V$;
  \item $\omega_{xy} > 0$, if $y \sim x$;
  \item $\omega_{xy} = 0$, if $(x, y) \notin E$,
\end{enumerate}
where we adopt the notation conventions $y \sim x$ if $(x, y) \in E$, and $\omega_{xy}:= \omega(x, y)$ for $(x,y)\in E$.
The vertex measure $\mu : V \to (0, \infty)$ is defined as $\mu(x) := \sum\limits_{y \sim x}\omega_{xy}$.
For any vertex subset $U\subset V$, we denote $E|_U := \{(x, y) \in E : x, y \in U\}$, $\partial U := \{y \in V\backslash U : \exists x \in U \ \text{s.t.}\  x \sim y, \}$, and $\overline{U} := U \cup\partial U$.
Let $C(V)$ be the set of real-valued functions on $V$.
For $1\leq p<\infty$, we define $\ell^p(V, \mu) := \left\{f \in C(V): \|f\|_{\ell^{p}(V, \mu)} < \infty\right\}$ with the norm
\begin{equation*}
  \|f\|_{\ell^p(V, \mu)} := \displaystyle\left(\sum_{x \in V}\mu(x)|f(x)|^p\right)^{\frac{1}{p}},
\end{equation*}
and $\ell^{\infty}(V, \mu) := \left\{f \in C(V): \|f\|_{\ell^{\infty}(V, \mu)} < \infty\right\}$ with the norm
\begin{equation*}
  \|f\|_{\ell^{\infty}(V, \mu)} := \displaystyle\sup_{x \in V} |f(x)|.
\end{equation*}
When no confusion may arise, we often abbreviate the aforementioned norms as $\|\cdot\|_{\ell^p}$ and $\|\cdot\|_{\ell^\infty}$ respectively.
We adopt the definition of the $p$-Laplacian operator given in \cite{amghibech2003}.
\begin{definition}\it
  The $p$-Laplacian of any function $f \in C(V)$ is defined as
  \begin{equation*}
    \Delta_p f(x) := \displaystyle\frac{1}{\mu(x)}\sum_{y \sim x}\omega_{xy}|f(y) - f(x)|^{p - 2}(f(y) - f(x)).
  \end{equation*}
  For any finite subset $U \subset V$ and $f \in C(\overline{U})$, we define
  \begin{align*}
    \Delta_{p}|_U(f)(x) = \displaystyle & \frac{1}{\mu(x)}\left(\sum_{\underset{y \in U}{y \sim x}} \omega_{xy}|f(y) - f(x)|^{p - 2}(f(y) - f(x))\right. \\
    &\qquad+ 
                                         \left. \sum_{\underset{y \in \partial U}{y \sim x}}\omega_{xy}|f(y) - f(x)|^{p - 2}(f(y) - f(x))\right).
  \end{align*}
\end{definition}

We adopt the definition of blow-up in \cite{Lin2018}.
\begin{definition}\it
  Let $u(x,t): V\times [0,T)\to \mathbb{R}$ be a solution of \eqref{eq_mubiao}.
  If there exists some vertex $x_0 \in V$ such that
  \begin{equation*}
    \lim\limits_{t\to T^-}|u(x_0, t)|=\infty,
  \end{equation*}
  we call the solution $u$ blows up in finite time.
\end{definition}

The next lemma ensures that the solution to \eqref{eq_mubiao} exists locally.
\begin{lemma}\it
  There exists some $T > 0$, such that the problem \eqref{eq_mubiao} has a solution $u(x,t): V\times [0,T]\to \mathbb{R}$, that is continuous with respect to $t\in [0,T]$ for any $x \in V$.
  \label{lem_mubiaocunz}
\end{lemma}
\begin{proof}
  We apply the Banach fixed-point theorem to prove the lemma.
  First, define a space
  \begin{equation*}
    \mathcal{F}_T := \left\{u: \forall x \in V, u(x, \cdot) \in C[0, T], \ \text{and}\ \|u\|_{\mathcal{F}_T} < \infty \right\},
  \end{equation*}
  where $T>0$ will be determined later and
  \begin{equation*}
    \|u\|_{\mathcal{F}_T} := \displaystyle\sup_{V\times [0, T]} |u|.
  \end{equation*}
  It follows by direct verification that $\mathcal{F}_T$, when endowed with the norm $\|\cdot\|_{\mathcal{F}_T}$, forms a Banach space.
  Regarding the initial value as a time-independent constant function, i.e., $u_0(x,t) = u_0(x)$, for any $(x,t)\in V\times [0,T]$, we have $u_0\in \mathcal{F}_T$.
  Define an operator $D$ as
  \begin{equation*}
    D[v] := u_0(x) + \displaystyle\int_{0}^{t}\Delta_p v(x, s){\rm d}s + \int_{0}^{t} \sigma(x, s)\Phi(v(x, s)){\rm d}s, \quad  x \in V,\ t \in [0, T].
  \end{equation*}
  It is obvious that for any $v \in \mathcal{F}_T$, $D[v](x, t) \in C[0, T]$ for $x \in V$.
  Therefore, $D$ is a well-defined operator from $\mathcal{F}_T$ to $\mathcal{F}_T$.

  Next we show that the operator $D$ is contractive in the closed ball $B_n\subset \mathcal{F}_T$, where $n\geq 2$ is a fixed integer, and
  \begin{equation}
    B_{n} := B(0, n\|u_0\|_{\mathcal{F}_T}) = \{v \in \mathcal{F}_T: \|v\|_{\mathcal{F}_T} \leq n\|u_0\|_{\mathcal{F}_T}\}.
    \label{eq_balldef}
  \end{equation}
  It is clear that $u_0\in B_n$.
  For any $u, v \in B_{n}$, we have
  \begin{align}\label{contractD}
    |(D[u] - D[v])(x, t)|
     & = \left|\displaystyle\int_{0}^{t} (\Delta_{p} u - \Delta_p v) + \sigma(x, s)(\Phi(u) - \Phi(v)){\rm d}s\right|
    \notag                                                                                                            \\
     & \leq \displaystyle\int_{0}^{t} |\Delta_{p} u - \Delta_p v|
    {\rm d}s + \int_{0}^{t} \sigma(x, s)|\Phi(u) - \Phi(v)|{\rm d}s.
  \end{align}
  Since the univariate function $g(\cdot) = |\cdot|^{p-2}\cdot: \mathbb{R} \to \mathbb{R}$ $(p>2)$ is locally Lipschitz continuous, and $u$, $v$ are uniformly bounded, there exists a constant $L_p = L_p(n, u_0, p) > 0$ independent of $T$ such that for all $x\in V$ and $s \in [0,T]$,
  \begin{align*}
    |\Delta_p u(x,s) - \Delta_p v(x,s)|
     & \leq \displaystyle\frac{L_p}{\mu(x)}\sum_{y \sim x}\omega_{xy}|u(y, s) - u(x, s) - (v(y, s) - v(x, s))|
    \\
     & \leq \displaystyle 2L_p\left(\sum_{y \sim x}\frac{\omega_{xy}}{\mu(x)}\|u - v\|_{\mathcal{F}_T}\right) 
      \leq 2L_p\|u - v\|_{\mathcal{F}_T}.
  \end{align*}
  By our assumptions (ii) and (iii) in Section \ref{introduction}, there exists a constant $L_{\Phi} = L_{\Phi}(n, u_0, T) > 0$ increasing with $T$ such that
  \begin{equation*}
    \sigma(x, s)|\Phi(u) - \Phi(v)|
    \leq L_{\Phi}\|u - v\|_{\mathcal{F}_T}, \qquad \forall(x, s) \in V\times [0, T].
  \end{equation*}
  Combining these estimates, we obtain from \eqref{contractD}
  \begin{equation}
    |(D[u] - D[v])(x, t)|
    \leq (2L_p + L_{\Phi})\|u - v\|_{\mathcal{F}_T} T.
    \label{eq_lianxu_u}
  \end{equation}
  Let $\rho:= (2L_p + L_{\Phi})T$.
  Since $L_p$ is independent of T and $L_{\Phi}$ increases with $T$, we can choose $T > 0$ such that $\rho \in (0,1)$ and $\rho \leq 1 - \frac{1}{n}$.
  Consequently we have
  \begin{equation*}
    \|D[u] - D[v]\|_{\mathcal{F}_T} \leq \rho\|u - v\|_{\mathcal{F}_T}.
  \end{equation*}
  Finally, since $\rho \leq 1 - \frac{1}{n}$, for any $v \in B_n$, we have
  \begin{equation*}
    \|D[v]\|_{\mathcal{F}_T} \leq \|D[v] - D[0]\|_{\mathcal{F}_T} + \|D[0]\|_{\mathcal{F}_T} \leq (\rho n + 1)\|u_0\|_{\mathcal{F}_T} \leq n\|u_0\|_{\mathcal{F}_T}.
  \end{equation*}
  Thus $D: B_n \to D(B_n) \subset B_n$ is a contractive operator.
  The Banach fixed-point theorem immediately tells us that the problem \eqref{eq_mubiao} admits one solution within the time interval $[0,T]$.
\end{proof}

To present the comparison principle, we need to introduce the following initial-boundary value problem:
\begin{align}
  \begin{cases}
    v_t - \Delta_p|_U v = \sigma(x, t)\Phi(v),\quad & (x,t)\in U\times (0,\infty),           \\
    v(x, t) = 0,                               & (x,t) \in \partial U\times [0,\infty), \\
    v(x,0) = u_0,                              & x \in U,
  \end{cases}
  \label{eq_fuzhu}
\end{align}
where $u_0$, $\sigma$ and $\Phi$ are the same as those functions in \eqref{eq_mubiao} and $U$ is a finite subset of $V$ such that the induced graph $(U, E|_U)$ is connected.
The next lemma confirms the local existence of the problem \eqref{eq_fuzhu}.
\begin{lemma}\it
  There exists a $T > 0$ such that the problem \eqref{eq_fuzhu} has a unique solution $u(x, t): \overline{U}\times [0, T] \to \mathbb{R}$, that is continuous with respect to $t \in [0, T]$ for any $x \in U$.
  \label{lem_cunzai}
\end{lemma}
\begin{proof}
  The proof follows a similar approach to that of Lemma \ref{lem_mubiaocunz}, with the key distinctions lying in the definitions of the function space and its associated norm, and the contractive operator.
  Specifically, we define
  \begin{equation*}
    \tilde{\mathcal{F}}_T := \left\{u: \forall x\in U, u(x, \cdot) \in C[0, T], \ \text{and}\  u(x, t) = 0 \text{ on } \partial U \times [0, T] \right\},
  \end{equation*}
  which is equipped with the norm
  \begin{equation*}
    \|u\|_{\tilde{\mathcal{F}}_T} := \displaystyle\max_{\overline{U}\times [0, T]} |u|.
  \end{equation*}
  Similarly, we define an operator $\tilde{D}: \tilde{\mathcal{F}}_T \to \tilde{\mathcal{F}}_T$ as
  \begin{align*}
    \tilde{D}[v] := \begin{cases}
                      u_0(x) + \displaystyle\int_{0}^{t}\Delta_p|_U v(x, s){\rm d}s + \int_{0}^{t} \sigma(x, s)\Phi(v(x, s)){\rm d}s, & x \in U,\ t \in [0, T],          \\
                      0,                                                                                                              & x \in \partial U,\ t \in [0, T].
                    \end{cases}
  \end{align*}
  Taking $u$ and $v$ from the ball $B_n$, in view of the finiteness of $\overline{U}$, we have
  \begin{align*}
    |(\tilde{D}[u] - \tilde{D}[v])(x, t)|
     & \leq \displaystyle\int_{0}^{t} \left|\Delta_{p}|_U u - \Delta_p|_U v\right| {\rm d}s + \int_{0}^{t} \sigma(x, s)|\Phi(u) - \Phi(v)|{\rm d}s \notag \\
     & \leq C\|u - v\|_{\tilde{\mathcal{F}}_T}T,
  \end{align*}
  where the constant $C$ increases with $T$.
  Therefore, we can choose a small $T>0$ such that $\tilde{D}: B_n \to B_n$ is contractive.
  Again the Banach fixed-point theorem gives the desired result.
\end{proof}

We now present a comparison principle, which plays a crucial role in proving the main results.
\begin{lemma}\it
  Let $U \subset V$ be a finite set of vertices and $T > 0$ (possibly $T = \infty$).
  Suppose $u(x, t)$ and $v(x, t) \in C^1([0, T))$ for each $x \in U$, and satisfy
  \begin{align}
    \begin{cases}
      u_t - \Delta_p|_U u - \sigma(x, t)\Phi(u) \geq v_t - \Delta_p|_U v - \sigma(x, t)\Phi(v), &(x, t) \in U \times (0, T),          \\
      u(x, t) \geq v(x, t),                                                                     &(x, t) \in \partial U \times [0, T), \\
      u(x, 0) \geq v(x, 0),                                                                     &x \in U.
    \end{cases}
    \label{eq_bijiao}
  \end{align}
  Then $u(x, t) \geq v(x, t)$ for all $(x, t) \in \overline{U} \times [0, T)$.
  \label{lem_bijiao}
\end{lemma}
\begin{proof}
  Since $U$ is a finite set, we have $\sup\limits_{\overline{U} \times [0, T^\prime]} (|u| + |v|) < \infty$, $\forall T^\prime \in (0, T)$.
  Then the locally Lipschitz continuity of $\Phi$ ensures the existence of a constant $L > 0$ such that
  \begin{equation*}
    |\Phi(u) - \Phi(v)|
    \leq L|u - v|, \forall (x, t) \in \overline{U} \times [0, T^\prime].
  \end{equation*}
  For a sufficiently large $\lambda > 0$ to be determined later, we define
  \begin{align*}
    \widetilde{u}(x, t) := \mathrm{e}^{-\lambda t}u(x, t), \qquad \widetilde{v}(x, t) := \mathrm{e}^{-\lambda t} v(x, t).
  \end{align*}
  Multiplying both sides of \eqref{eq_bijiao} by $\mathrm{e}^{-\lambda t}$, we obtain
  \begin{equation}\label{utildeinq}
    \widetilde{u}_t - \widetilde{v}_t - \mathrm{e}^{\lambda(p - 2)t}(\Delta_p|_U\widetilde{u} - \Delta_p|_U\widetilde{v}) + \lambda(\widetilde{u} - \widetilde{v}) - \mathrm{e}^{-\lambda t}\sigma(x, t)(\Phi(u) - \Phi(v)) \geq 0
  \end{equation}
  in $U \times (0, T^\prime]$. Due to the compactness of $\overline{U} \times [0, T^\prime]$, there exists some $(x_0, t_0)\in \overline{U} \times [0, T^\prime]$ such that
  \begin{equation*}
    (\widetilde{u} - \widetilde{v})(x_0, t_0) = \displaystyle\min_{\overline{U}\times[0, T^\prime]} (\widetilde{u} - \widetilde{v})(x, t).
  \end{equation*}

  Now we claim that $(\widetilde{u} - \widetilde{v})(x_0, t_0) \geq 0$. If not, $(x_0, t_0)$ must belong to $U \times (0, T^\prime]$. Since $(x_0, t_0)$ is a minimum point, we have
  \begin{equation}\label{tder}
    \widetilde{u}_t(x_0, t_0)-\widetilde{v}_t(x_0, t_0) \leq 0
  \end{equation}
  and
  \begin{equation}\label{uvxy}
    \widetilde{u}(x_0, t_0) - \widetilde{v}(x_0, t_0) \leq \widetilde{u}(y, t_0) - \widetilde{v}(y, t_0),\qquad \forall y \in \overline{U}.
  \end{equation}
  Combining \eqref{tder} and \eqref{uvxy}, and noticing that $g(\cdot) = |\cdot|^{p-2}\cdot: \mathbb{R} \to \mathbb{R}$ $(p>2)$ is increasing on $\mathbb{R}$, we get
\begin{align*}
    (\Delta_p|_U\widetilde{u} - \Delta_p|_U\widetilde{v})(x_0, t_0) & = \displaystyle\frac{1}{\mu(x_0)}\sum_{y \sim x_0}\omega_{x_0y}\left(g\left(\widetilde{u}(y, t_0) - \widetilde{u}(x_0, t_0)\right) - g\left(\widetilde{v}(y, t_0) - \widetilde{v}(x_0, t_0)\right)\right) \geq 0.
  \end{align*}
  This immediately yields $-\mathrm{e}^{\lambda(p - 2)t}(\Delta_p|_U\widetilde{u} - \Delta_p|_U\widetilde{v})(x_0, t_0) \leq 0$.
  By choosing $\lambda$ large enough such that $\lambda \geq L\max\limits_{\overline{U} \times [0, T^\prime]} \sigma + 1$, we can derive the following estimates:
  \begin{align*}
    &\lambda(\widetilde{u} - \widetilde{v})(x_0, t_0) + \mathrm{e}^{-\lambda t_0}\sigma(x_0, t_0)(\Phi(v(x_0, t_0)) - \Phi(u(x_0, t_0)))                                           \\
                                                      \leq& \lambda(\widetilde{u} - \widetilde{v})(x_0, t_0) + L\sigma(x_0, t_0)|\widetilde{u}(x_0, t_0) - \widetilde{v}(x_0, t_0)|
    \\
                                                      =& (\lambda - L\sigma(x_0, t_0))(\widetilde{u} - \widetilde{v})(x_0, t_0) < 0.
  \end{align*}
  However, this contradicts the inequality \eqref{utildeinq}.
  Hence
  \begin{equation*}
    \begin{array}{ll}
      \widetilde{u}(x, t) \geq \widetilde{v}(x, t), & \text{in } \overline{U}\times [0, T^\prime],
    \end{array}
  \end{equation*}
  which implies $u \geq v$ in $\overline{U}\times [0, T^\prime]$.
  Due to the arbitrariness of $T^\prime$, we have
  \begin{equation*}
    \begin{array}{ll}
      u(x, t) \geq v(x, t), & \text{in } \overline{U}\times[0, T).
    \end{array}
  \end{equation*}
  This completes the proof of the lemma.
\end{proof}

By applying Lemma \ref{lem_bijiao}, we immediately deduce the uniqueness of solutions for the equation \eqref{eq_fuzhu}, as stated in the following corollary.
This uniqueness result guarantees that any local solution to the problem \eqref{eq_fuzhu} can be extended to its unique maximal time of existence.
\begin{Corollary}\it
  If $u$ and $v$ are solutions of the problem \eqref{eq_fuzhu} in the time interval $[0,T)$, then $u \equiv v$ in $\overline{U} \times [0, T)$.
  \label{cor_weiyi}
\end{Corollary}

To complete the preliminary material, we now turn to the eigenvalue problem with the Dirichlet boundary.
\begin{definition}\it
  For a finite subset $U \subset V$ with boundary $\partial U\neq \emptyset$, we call
  \begin{align}
    \begin{cases}
      -\Delta_p \varphi = \lambda |\varphi|^{p - 2}\varphi, & x \in U,         \\
      \varphi = 0,                                          & x \in \partial U
    \end{cases}
    \label{eq_tezheng}
  \end{align}
  the eigenvalue problem for $-\Delta_p$ with Dirichlet boundary.
  \label{def_tezheng}
\end{definition}

Let $\ell^p_0(U)$ be the function space consisting of functions with support contained in a bounded domain $U$
\begin{equation*}
  \ell^p_0(U) := \{f \in \ell^p(\overline{U}): f|_{\partial U} \equiv 0\}.
\end{equation*}
The sphere in $\ell^p_0(U)$ is denoted by $\mathcal{S}_p := \{f \in \ell_0^{p}(U): \|f\|_{\ell^p} = 1\}$.
According to \cite{Tudisco2018}, we present the definitions of a strong nodal domain.
\begin{definition}\it
  \( \forall f\in C(V) \), if the subgraph \( (A, E|_A) \) induced on a subset \( A\subset V \) is a maximal connected component of either \( \{x : f(x) > 0\} \) or \( \{x : f(x) < 0\} \), we call \( A \) a strong nodal domain of \( f \).
\end{definition}

We begin by demonstrating the existence of solutions to the eigenvalue problem \eqref{eq_tezheng}.
\begin{lemma}\it
  There exists $\lambda \in \mathbb{R}$ such that the problem \eqref{eq_tezheng} has nontrivial solutions.
  Moreover, if the induced graph $G(U) = (U, E|_U)$ is connected, the smallest eigenvalue $\lambda_1 > 0$.
  \label{lem_tezheng}
\end{lemma}
\begin{proof}
  By Definition \ref{def_tezheng}, $\Delta_p$ is a nonlinear operator acting on the linear space $\ell^p_0(U)$.
  For any nontrivial function $f \in \ell^p_0(U)$ ($f\not\equiv 0$), we define the Rayleigh quotient as
  \begin{equation*}
    \mathcal{R}_p(f) = \displaystyle\frac12\frac{\sum\limits_{x, y \in \overline{U}} \omega_{xy} |f(y) - f(x)|^{p}}{\sum\limits_{x \in U} \mu(x) |f(x)|^{p}}.
  \end{equation*}
  It is obvious that the infimum
  \begin{equation*}
    \lambda_1 = \inf_{\substack{f \in \ell^p_0(U) \\ f \not\equiv 0}} \mathcal{R}_p(f)
  \end{equation*}
  is well-defined and non-negative.
  Consider the functional
  \begin{equation*}
    \mathcal{I}_p(f) := \frac{1}{2} \sum_{x, y \in \overline{U}} \omega_{xy} |f(y) - f(x)|^p,\qquad f\in \ell^p_0(U)
  \end{equation*}
  and the constraint set $\mathcal{S}_p = \{f \in \ell_0^{p}(U): \|f\|_{\ell^p} = 1\}$.
  Then $\lambda_1 = \inf\limits_{f \in \mathcal{S}_p} \mathcal{I}_p(f)$. Since $\mathcal{S}_p$ is compact in the finite-dimensional space $\ell^p_0(U)$ and $\mathcal{I}_p$ is continuous on $\mathcal{S}_p$, there exists a minimizer $\varphi_1 \in \mathcal{S}_p$ such that $\mathcal{I}_p(\varphi_1) = \lambda_1$.
  Every function $f\in\ell^p_0(U)$ vanishes on the boundary $\partial U$, and $f$ can be identified with a vector in $\mathbb{R}^{|U|}$.
  Therefore, $\varphi_1$ is the solution to the constrained optimization problem defined on $\mathbb{R}^{|U|}$ subject to the constraint $\mathcal{S}_p(f)=\|f\|_{\ell^p}-1=0$.
  Consequently, the gradients $\nabla \mathcal{I}_p(\varphi_1)$ and $\nabla \mathcal{S}_p(\varphi_1)$ must be collinear.
  Taking any test function $\psi \in \ell^p_0(U)$ for the functional $\mathcal{I}_p$, we have
  \begin{align}
    \frac{\mathrm{d}}{\mathrm{d}t}\bigg|_{t=0} \mathcal{I}_p(\varphi_1 + t\psi) & = \frac{p}{2} \sum_{x, y \in \overline{U}} \omega_{xy} |\varphi_1(y) - \varphi_1(x)|^{p-2} (\varphi_1(y) - \varphi_1(x)) (\psi(y) - \psi(x)) \notag \\
                                                              & = -p \sum_{x \in U} \mu(x) \psi(x) \Delta_p \varphi_1(x),
    \label{eq:deriv_objective}
  \end{align}
  where the last equality uses $\psi|_{\partial U} = 0$.
  Since this identity holds for any $\psi\in \ell_0^p$, we have
  \begin{equation*}
    \nabla \mathcal{I}_p(\varphi_1) = \left(-p \mu(x) \Delta_p \varphi_1(x)\right)_{x \in U}.
  \end{equation*}
  The constraint can be written as $\mathcal{S}_p(f)= \sum\limits_{x \in U}\mu(x)|f(x)|^p - 1$ and its gradient at $\varphi_1$ is
  \begin{equation*}
    \nabla \mathcal{S}_p(\varphi_1) = \left(p \mu(x) |\varphi_1(x)|^{p-2} \varphi_1(x)\right)_{x \in U}.
  \end{equation*}
  $\|\varphi_1\|_{\ell^p} = 1$ implies that there exists some $x_0 \in U$ with $\varphi_1(x_0) \neq 0$.
  Hence
  \begin{equation*}
p \mu(x_0) |\varphi_1(x_0)|^{p-2} \varphi_1(x_0) \neq 0.
\end{equation*}
  Consequently $\nabla \mathcal{S}_p(\varphi_1) \neq 0$.
  By the Lagrange multiplier theorem, there exists $\lambda \in \mathbb{R}$ such that
  \begin{equation*}
    \nabla \mathcal{I}_p(\varphi_1) = \lambda \nabla \mathcal{S}_p(\varphi_1).
  \end{equation*}
  Thus, $\varphi_1$ satisfies the eigenvalue equation
  \begin{equation*}
    -\Delta_p \varphi_1(x) = \lambda |\varphi_1(x)|^{p-2} \varphi_1(x), \quad x \in U.
  \end{equation*}
  Since $\varphi_1$ achieves the minimum of $\mathcal{R}_p$, we have $\lambda = \lambda_1$.
  It remains to show that $\lambda_1 > 0$ if the induced graph $G(U) = (U, E|_U)$ is connected.
  If not, we have $\lambda_1 = 0$. Then the corresponding eigenfunction $\varphi_1$ satisfies $\mathcal{R}_p(\varphi_1) = 0$, which implies
  \begin{equation*}
    \sum_{x, y \in \overline{U}} \omega_{xy} |\varphi_1(y) - \varphi_1(x)|^p = 0.
  \end{equation*}
  Since the graph $G(U)$ is connected, this forces $\varphi_1\equiv c$ for some constant $c$.
  Then the boundary condition $\varphi_1|_{\partial U} \equiv 0$ implies that $\varphi_1 \equiv 0$, which contradicts $\|\varphi_1\|_{\ell^p} = 1$.
  Hence, $\lambda_1 > 0$.
\end{proof}

The following corollary demonstrates the particular significance of the first eigenvalue and its eigenfunction in our problem.
\begin{Corollary}\it
  If the induced subgraph $G(U)$ is connected, then the eigenfunction $\varphi_1$ of the first eigenvalue $\lambda_1$ does not change sign in $U$.
  \label{cor_tezheng}
\end{Corollary}
\begin{proof}
  Since $\varphi_1 \not\equiv 0$, it has at least one strong nodal domain $A \subset U$.
  It remains to show that $A = U$.

  For any $f \in \ell_0^{p}(U)$, one has
  \begin{equation*}
    \left||f(y)|
    - |f(x)|\right| \leq |f(y) - f(x)|,
  \end{equation*}
  with equality holding if and only if $f(x)f(y) \ge 0$ for any $x, y \in \overline{U}$.
  Notice that replacing $f(y)$ and $f(x)$ with $|f(y)|$ and $|f(x)|$ in the numerator does not increase the value of $\mathcal{R}_p$.
  Moreover, when  $f(x)$ and $|f(y)|$ have opposite signs, this replacement shall strictly decrease the numerator.
  Therefore, without loss of generality, we may assume $\varphi_1 \ge 0$ on $U$.
  Assume that there exists $y_0 \in U$ such that $\varphi_1(y_0) = 0$.
  Since $G(U)$ is connected, there exists a vertex $z_0$ in the strong nodal domain $A$ such that $\varphi_1(z_0) > 0$.
  Fix a simple path $\Gamma = (v_0, v_1, \cdots, v_m)$ in $G(U)$ with $v_0 = y_0$ and $v_m = z_0$.
  Let $v_i$ be the first vertex along $\Gamma$ (in the order $v_0, v_1, \cdots, v_m$) satisfying $\varphi_1(v_i) = 0$ with $\varphi_1(v_{i+1}) > 0$, where $v_{i+1}$ denotes the next vertex along $\Gamma$.
  Obviously, the vertex set
  \begin{equation*}
    W_{v_i} := \{ x \sim v_i : \varphi_1(x) > 0 \}
  \end{equation*}
  is nonempty.

  Define
  \begin{align*}
    \varphi_{\varepsilon}(x) := \begin{cases}
                                  \varepsilon, & x = v_i,          \\
                                  \varphi_1(x),  & \text{otherwise},
                                \end{cases}
  \end{align*}
  where $\varepsilon>0$ is a constant to be determined later.
  Recall that the numerator of $\mathcal{R}_p$ is
  \begin{equation*}
    N(f) := \sum_{x,y\in \overline{U}} \omega_{xy} |f(y) - f(x)|^p.
  \end{equation*}
  Since $\varphi_\varepsilon$ coincides with $\varphi_1$ everywhere except at $v_i$, it follows that
  \begin{equation*}
    N(\varphi_\varepsilon) = N(\varphi_1) + N_{\delta},
  \end{equation*}
  where
  \begin{equation*}
    N_{\delta} = \displaystyle 2\Bigg(\sum_{x \in W_{v_i}}\omega_{v_i x}(|\varepsilon - \varphi_1(x)|^p - |\varphi_1(x)|^p) + \sum_{x \not\in W_{v_i}}\omega_{v_i x}\varepsilon^p\Bigg).
  \end{equation*}
  For $a > 0$ one has $|\varepsilon - a|^p = |a|^p - p |a|^{p-1} \varepsilon + o(\varepsilon)$.
  Since $p>2$, there exists some constant $c>0$ depending on $p$ and degree of $v_i$ such that $N_{\delta} = -c\varepsilon + o(\varepsilon)$.
  Moreover, the denominator of $\mathcal{R}_p$ increases by $\sum_{x \in U} \mu(x)(|\varphi_{\varepsilon}(x)|^p-|\varphi_1(x)|^p) = \mu(v_i)\varepsilon^p = o(\varepsilon)$.
  Hence we can choose a sufficiently small $\varepsilon$, such that
  \begin{equation*}
    \mathcal{R}_p(\varphi_{\varepsilon}) < \mathcal{R}_p(\varphi_1).
  \end{equation*}
  This contradicts that $\varphi_1$ achieves the minimum of $\mathcal{R}_p$.
  Therefore $\varphi_1$ does not change sign on $U$.
\end{proof}

\section{Proofs of the main theorems}
\label{sec:blowup}

In this section, we turn to the proof of the main results.
We first establish the following lemma, whose proof follows the strategy developed in~\cite{Chung2014a}.
\begin{lemma}\it
  Assume \eqref{eq_HP} holds and $u_0 \geq 0$ but is not identically zero.
  Then the solution of problem \eqref{eq_fuzhu} blows up in finite time in $U$.
  \label{lem_fuzhu_baopo}
\end{lemma}

\begin{proof}
  Let $u(x,t)$ be a solution of \eqref{eq_fuzhu} in $U\times [0,T)$.
  Denote $u(t) := \max\limits_{x \in U} u(x, t)$ and $\overline{u}_0 = \max\limits_{x \in U} u_0(x)$.
  For any $x \in U$, since $u_t(x, t)$ is continuous in $t$ and integrable, it follows that the function $t \mapsto u(x,t)$ is absolutely continuous.
  Noticing that $U$ is finite and the maximum of finitely many absolutely continuous functions is again absolutely continuous, we know that $u(t)$ is absolutely continuous on $[0,T^\prime]$ and thus differentiable almost everywhere on $[0, T^\prime]$ for any $T^\prime\in (0,T)$.
  Choose $t_0 \in [0,T^\prime]$ such that $u_t(t_0)$ exists. By the definition of $u(t_0)$, there exists $x_0 \in U$ such that $u(x_0,t_0) = u(t_0)$.
  We claim that $u_t(t_0) \geq 0$. If not, we have
  \begin{equation*}
    u_t(x_0, t_0) = -\frac{1}{\mu(x_0)} \sum_{y \sim x_0} \omega_{x_0y} \left|u(x_0, t_0) - u(y, t_0)\right|^{p - 1} + \sigma(x_0, t_0) \Phi(u(x_0, t_0)).
  \end{equation*}
  Hence
  \begin{align*}
    u_t(x_0,t_0) & \geq -\frac{1}{\mu(x_0)} \sum_{y \sim x_0} \omega_{x_0 y} u(x_0, t_0)^{p - 1} + \sigma(x_0, t_0) \Phi\left(u(x_0, t_0)\right) \\
                 & = -u^{p - 1}(x_0, t_0) + \sigma(x_0, t_0) \Phi\left(u(x_0, t_0)\right).
  \end{align*}
  By assumptions (i) and (ii) in Section \ref{introduction}, there exists $\delta > 0$ such that $\sigma(x, t) \geq \delta$ on $U\times [0, T^\prime]$.
  By \eqref{eq_HP}, we can take $C_0 = \delta^{-1}$ such that
  \begin{equation*}
    \delta \Phi(s) - s^{p - 1} > (\delta C_0 - 1)s^{p - 1} = 0 \quad\text{for}\  s\geq 0.
  \end{equation*}
  Therefore
  \begin{equation}
    u_t(x_0, t_0) \geq \delta \Phi(u(x_0, t_0)) - u^{p - 1}(x_0, t_0) > 0,
    \label{eq_ut_ge0}
  \end{equation}
  contradicting $u_t(t_0) = u_t(x_0,t_0) < 0$.
  Thus $u_t(t_0) \geq 0$ whenever it exists, and $u(t)$ is nondecreasing on $[0, T^\prime]$.
  Since $T^\prime$ can be chosen arbitrarily from $[0,T)$, $u(t)$ is nondecreasing on $[0, T)$. Hence $u(t)\geq \overline{u}_0$ for all $t\in[0,T)$.
  By arguments similar to those of \eqref{eq_ut_ge0}, we have for almost every $t \in [0,T)$,
  \begin{equation*}
    u_t(t) \geq \delta \Phi(u(t)) - u^{p - 1}(t) > 0.
  \end{equation*}
  Hence
  \begin{equation}
    \displaystyle\frac{u_t(t)}{-u^{p - 1}(t) + \delta \Phi(u(t))} \geq 1 \ \  \text{a.e.} \ t \in [0, T)
    \label{eq_dayu1}
  \end{equation}
  Integrating \eqref{eq_dayu1} from $0$ to $t$ gives
  \begin{equation}
    t \leq \displaystyle\int_{0}^{t}\frac{u_t(s){\rm d}s}{-u^{p - 1}(s) + \delta \Phi(u(s))} = \int_{\overline{u}_0}^{u(t)}\frac{{\rm d}s}{-s^{p - 1} + \delta \Phi(s)}.
    \label{eq_jifen}
  \end{equation}
  Define
  \begin{equation*}
    F(x) := \displaystyle\int_{x}^{\infty}\frac{{\rm d}s}{-s^{p - 1} + \delta \Phi(s)}.
  \end{equation*}
  Since $\overline{u}_0 > 0$ and $p > 2$, $F(x)$ is finite for any $x>\overline{u}_0$ and strictly decreasing on $(\overline{u}_0,\infty)$.
  From \eqref{eq_jifen}, it follows that
  \begin{equation*}
    t \leq F(\overline{u}_0) - F(u(t)).
  \end{equation*}
  Hence we obtain that
  \begin{equation*}
    F(u(t)) \leq F(\overline{u}_0) - t.
  \end{equation*}
  If the solution exists on $[0,\infty)$, as $t \to F(\overline{u}_0)$ from below, it yields that $F(u(t)) \to 0$, which implies $u(t) \to \infty$.
  This contradicts the finiteness of $u(x, t)$ on any finite time interval $[0,T)$.
  Therefore the maximal existence time $T$ must satisfy $T \leq F(\overline{u}_0) < \infty$, and $u(x,t)$ blows up in finite time.
\end{proof}

Now we can prove Theorem \ref{th_blowup_gpm1}.

\begin{proof}[Proof of Theorem \ref{th_blowup_gpm1}]
  Let $u$ be a nonnegative solution of \eqref{eq_mubiao}.
  Set $C_0 := \delta^{-1}$ as in Lemma~\ref{lem_fuzhu_baopo}. Since $u_0 \not\equiv 0$, there exists $x_0 \in V$ such that $u_0(x_0) > 0$.
  Choose a finite subset $U \subset V$ with $x_0 \in U$ and the induced graph $(U, E|_U)$ is connected.
  Fix $T > 0$ so that $u$ is defined on $\overline{U} \times [0,T)$.
  Restricting \eqref{eq_mubiao} to $U$ gives
  \begin{align*}
    \begin{cases}
      u_t - \Delta_p|_U u \geq \sigma(x, t)\Phi(u), & (x, t) \in U \times (0, T),         \\
      u(x, t) \geq 0,                               & (x, t) \in \partial U\times [0, T), \\
      u(x, 0) = u_0(x),                             & x \in U.
    \end{cases}
  \end{align*}
  Let $v$ be the solution of the auxiliary problem \eqref{eq_fuzhu}.
  By Lemma \ref{lem_fuzhu_baopo}, the function $v$ blows up in finite time.
  Applying the comparison principle (Lemma \ref{lem_bijiao}), we obtain $u \ge v$ in $\overline{U} \times [0,T)$, hence $u$ also blows up in finite time.
  Theorem \ref{th_blowup_gpm1} is proved.
\end{proof}

\begin{remark}
  (i) In the proof of Theorem \ref{th_blowup_gpm1}, we can choose $U = \{x_0\}$ and $\partial U = \{y \in V : y \sim x_0\}$, where $x_0$ satisfies $u_0(x_0)>0$.
  The arguments in the proof imply that the solution $u$ blows up in finite time at every vertex $x_0$ with non-zero initial data $u_0(x_0)\neq 0$.
  (ii) Consider the case where the inequality in \eqref{eq_mubiao} becomes an equality, namely
  \begin{align}
    \begin{cases}
      u_t - \Delta_p u = \sigma(x, t)\Phi(u), & x \in V, t > 0, \\
      u(x, 0) = u_0(x),                       & x \in V, t = 0.
    \end{cases}
    \label{eq_mubiao=}
  \end{align}
  It is easy to check that Lemmas \ref{lem_fuzhu_baopo} and \ref{lem_bijiao} remain valid in this setting.
  Repeating the arguments used in the proof of Theorem \ref{th_blowup_gpm1} shows that the solution to problem \eqref{eq_mubiao=} also blows up in finite time under the same assumptions as in Theorem \ref{th_blowup_gpm1}.
\end{remark}

To prove Theorem \ref{th_blowup_g1lpm1}, we first follow the construction method in \cite{Kong2007} and consider the eigenvalue problem
\begin{align*}
  \begin{cases}
    -\Delta_p|_U \varphi = \lambda|\varphi|^{p - 2}\varphi(x), & x \in U,          \\
    \varphi(x) = 0,                                            & x \in \partial U,
  \end{cases}
\end{align*}
where $U$ is a finite subset of $V$ such that its induced graph $(U, E|_U)$ is connected.
Let $\lambda_1$ be the first eigenvalue and let $\varphi_1$ be the corresponding eigenfunction, which is normalized so that $\|\varphi_1\|_{\ell^\infty(U)} \leq 1$.
By Corollary \ref{cor_tezheng}, we may assume that $\varphi_1 > 0$ on $U$. Set $\varphi_{min} := \min\limits_{x \in U}\varphi_1 > 0$.
For any small $\varepsilon > 0$, consider the Cauchy problem
\begin{align}
  \begin{cases}
    \displaystyle\frac{{\rm d}h}{{\rm d}t} = -\lambda_1\|\varphi_1\|_{\ell^{\infty}(U)}^{p - 2}h^{p - 1} + C_1\delta\varphi_{min}^qh^q,                             & t > 0, \\
    h(0) = h_0 := \displaystyle\left(\frac{\lambda_1\|\varphi_1\|^{p - 2}_{\ell^{\infty}(U)}}{C_1\delta\varphi_{min}^q}\right)^{\frac{1}{q - p + 1}} + \varepsilon. &        \\
  \end{cases}
  \label{eq_blowup_cauchy}
\end{align}
As \eqref{eq_blowup_cauchy} is a separable nonlinear ordinary differential equation with $q > p-1$, its unique solution $h$ necessarily blows up in finite time.
Define $\underline{v}(x, t) := h(t)\varphi_1(x)$. Since $\varphi_1 > 0$ in $U$, the function $\underline{v}$ also blows up in finite time.
Then we state the following Lemma.
\begin{lemma}\it
  Assume $\min\limits_{x \in U} u_0$ is sufficiently large and let $v$ be a solution of \eqref{eq_fuzhu}.
  Then $v \geq \underline{v}$ in $\overline{U} \times [0, T)$. In particular, $v$ blows up in finite time.
  \label{lem_vxiajie}
\end{lemma}
\begin{proof}
  For any $x \in U$ and $t > 0$, we compute
  \begin{align*}
    \Delta_p|_U \underline{v} + \sigma(x, t)\Phi(\underline{v}) \geq& -\lambda_1|\varphi_1(x)|^{p - 2}\varphi_1(x)h^{p - 1}(t) + C_1\delta\varphi_{min}^qh^q \\
                               \geq& (-\lambda_1h^{p - 1}(t)\|\varphi_1\|_{\ell^{\infty}(U)}^{p - 2} + C_1\delta\varphi_{min}^{q}h^q)\varphi_1(x)             \\
                               =& \varphi_1(x)\displaystyle\frac{{\rm d}h}{{\rm d}t} = \underline{v}_t,
  \end{align*}
  where the first inequality uses \eqref{eq_HP2} with $s = \underline{v} = h(t)\varphi_1(x)$ and $\varphi_1(x) \geq \varphi_{min} >0$, while the second uses $\varphi_1(x) \leq \|\varphi_1\|_{\ell^\infty(U)} \leq 1$.
  Thus
  \begin{equation}
    \underline{v}_t - \Delta_p|_U \underline{v} - \sigma(x, t)\Phi(\underline{v}) \leq 0, \quad \ x \in U,\ t > 0.
    \label{eq_bijiaoxia1}
  \end{equation}
  On the boundary, we have $\underline{v} = v = 0$ on $\partial U \times [0,T)$.
  Furthermore, if
  \begin{equation*}
    \displaystyle\min_{x \in U} u_0(x) \geq h_0,
  \end{equation*}
  where $h_0$ is defined as in \eqref{eq_blowup_cauchy}, then $\underline{v}(x,0) \leq u_0(x) = v(x,0)$ holds for all $x \in U$.
  The inequality \eqref{eq_bijiaoxia1} and the comparison principle (Lemma \ref{lem_bijiao}) imply $v \geq \underline{v}$.
  Since $\underline{v}$ blows up in finite time, so does $v$. This completes the proof of Lemma \ref{lem_vxiajie}.
\end{proof}

We now complete the proof of Theorem~\ref{th_blowup_g1lpm1}.

\begin{proof}[Proof of Theorem \ref{th_blowup_g1lpm1}]
  Since $\|u_0\|_{\ell^\infty(V)}$ is sufficiently large, we can choose $x_0 \in V$ such that $u_0(x_0) \geq h_0$, where $h_0$ is defined in \eqref{eq_blowup_cauchy}.
  Let $U = \{x_0\}$, then $\partial U = \{y: y \sim x_0\}$.
  By Corollary \ref{cor_tezheng} and Lemma \ref{lem_vxiajie}, the solution $v$ of the auxiliary problem \eqref{eq_fuzhu} blows up in finite time.
  By the comparison principle \ref{lem_bijiao}, the solution $u$ of \eqref{eq_mubiao} also blows up in finite time. This proves Theorem \ref{th_blowup_g1lpm1}.
\end{proof}
\begin{remark}
  Similarly, Lemmas \ref{lem_fuzhu_baopo} and \ref{lem_bijiao} remain valid under the assumptions as in Theorem \ref{th_blowup_g1lpm1}.
  Repeating the arguments used in the proof of Theorem \ref{th_blowup_g1lpm1} shows that the solution of \eqref{eq_mubiao=} also blows up in finite time in this case.
\end{remark}



\section*{Acknowledgments}
This research is supported by the National Natural Science Foundation of China (No. 12271039) and the Open Project Program (No. K202303) of Key Laboratory of Mathematics and Complex Systems, Beijing Normal University.



\end{document}